\documentclass{elsarticle}
\usepackage{amsmath,amsthm,amssymb,url}

\newtheorem{theorem}{Theorem}
\newtheorem{lemma}[theorem]{Lemma}

\providecommand{\abs}[1]{\lvert#1\rvert}

\newcommand{\dist}{\mbox{\rm dist\/}}


\begin{document}

\begin{frontmatter}

\title{Metric Spaces in Which Many Triangles Are Degenerate}

\author[fn1]{Va\v sek Chv\' atal}
\author[fn4]{No\' e de Rancourt}
\author[fn2]{Guillermo Gamboa Quintero}
\author[fn3]{Ida Kantor}
\author[fn5]{P\' eter G.N. Szab\' o}

\fntext[fn1]{Department of Computer Science and Software Engineering, Concordia University, Montreal (Emeritus) and 
 Department of Applied Mathematics, Charles University, Prague (Visiting).
Supported by NSERC  grant RGPIN/5599-2014 and by H2020-MSCA-RISE project CoSP-GA No. 823748.
\\ {\em E-mail address:} {\urlstyle{tt} chvatal@cse.concordia.ca}}

\fntext[fn4]{ {Univ. Lille, CNRS, UMR 8524 – Laboratoire Paul Painlevé, F-59000 Lille, France. Supported by the Labex CEMPI (ANR-11-LABX-0007-01).\\ \em E-mail address:} {\urlstyle{tt} nderancour@univ-lille.fr} }

\fntext[fn2]{ {Computer Science Institute of Charles University, Prague.
Supported by GA\v{C}R grant 2-17398S. 
\\ \em E-mail address:} {\urlstyle{tt} gamboa@iuuk.mff.cuni.cz}}

\fntext[fn3]{ {Computer Science Institute of Charles University, Prague.
Supported by GA\v{C}R grant 22-19073S and by Charles University project UNCE/SCI/004.
\\ \em E-mail address:} {\urlstyle{tt} ida@iuuk.mff.cuni.cz}} 

\fntext[fn5]{Department of Engineering Mathematics, Faculty of Engineering and Information Technology, University of P\'ecs, P\'ecs, Hungary and Department of Computer Science and Information Theory, Budapest University of Technology and Economics, Hungary
\\ {\em E-mail address:} {\urlstyle{tt} szape@cs.bme.hu}}

\begin{abstract}
Richmond and Richmond ({\sc American Mathematical Monthly} {\bf 104} (1997), 713--719) proved the following theorem: If, in a metric space with at least five points, all triangles are degenerate, then the space is isometric to a subset of the real line. We prove that the hypothesis is unnecessarily  strong: In a metric space on $n$ points, fewer than $7n^2/6$ suitably placed degenerate triangles suffice. However, fewer than $n(n-1)/2$ degenerate triangles, no matter how cleverly placed, never suffice. 
\end{abstract}
\end{frontmatter}

Given a metric space $(V,\dist)$, we follow \cite{ACH} in writing $[rst]$ to signify that 
$r,s,t$ are pairwise distinct points of $V$ and $\dist(r,s)+\dist(s,t)=\dist(r,t)$. Following~\cite{RiRi97}, we refer to three-point subsets of $V$ as {\em triangles;\/} if $[rst]$, then the triangle $\{r,s,t\}$ is called {\em degenerate.\/} If a metric space is isometric to a subset of the real line with the usual Euclidean metric, then all its triangles are degenerate. Richmond and Richmond~\cite{RiRi97} proved the converse under a mild lower bound on the number of points in the metric space:
\begin{theorem}[Richmond and Richmond]\label{riri}
If, in a metric space with at least five points, all triangles are degenerate, then this metric space is isometric to a subset of the real line with the usual Euclidean metric.
\end{theorem}
Here, the lower bound on the number of points cannot be relaxed: consider the distances defined on $\{a,b,c,d\}$ by $\dist(a,b)\!=\!dist(b,c)\!=\!\dist(c,d)\!=\!\dist(d,a)\!=\!1$ and $\dist(a,c)\!=\!\dist(b,d)\!=\!2$.

In the finite case, Szab\' o~\cite[Theorem 1]{Sza22} proved that the hypothesis of Theorem~\ref{riri} can be relaxed as soon as the space has six points and that it can be relaxed further and further as the number of points gets larger and larger. 
\begin{theorem}\label{sza}
If, in a metric space with $n$ points such that $n\ge 6$, at most $n-4$ triangles are nondegenerate, 
then either this metric space is isometric to a subset of the real line with the usual Euclidean metric 
or precisely $n-4$ triangles are nondegenerate and either their intersection consists of two points or $n=6$ and the two triangles are disjoint.
\end{theorem} 
Actually, Theorem 1 of \cite{Sza22}, in addition to dealing also with smaller $n$, provides a complete description of metric spaces with $n$ points and precisely $n-4$ nondegenerate triangles. This description is irrelevant to our discussion except for its special case $n=5$. In making note of its corollary (which is our Lemma~\ref{sza2}), we shall call a four-point metric space {\em linear\/} if it is isometric to a subset of the real line with the usual Euclidean metric. (Readers beware: our meaning of the term ``linear'' is different from that of ~\cite[Definition 1]{Sza22}.) If all four triangles in a four-point metric space  are degenerate and yet the space is not linear, then we shall call it {\em circular\/} (since its points can be labeled as $0,1,2,3$ in such a way that $[012]$, $[123]$, $[230]$, $[301]$). 
\begin{lemma}\label{sza2}
If, in a metric space with five points, precisely one triangle is nondegenerate, 
then this metric space has a four-point circular subspace.
\end{lemma}
Let us call a set $A$ of three-point subsets of a set $V$ an {\em anchor in $V$\/} if, for every metric space $(V,\dist)$, the assumption that all triangles in $A$ are nondegenerate implies that $(V,\dist)$ is isometric to a subset of the real line with the usual Euclidean metric. In this terminology, Theorems~\ref{riri} and~\ref{sza} show that every set of $\binom{n}{3}-n+5$ three-point subsets of an $n$-point set $V$ is an anchor in $V$. The purpose of this note is to point out that there are far smaller anchors:
\begin{theorem}\label{gg}
In every $n$-point set where $n\neq 4$, there is an anchor of fewer than $\frac{7}{6}n^2$ sets. 
\end{theorem}
\noindent (The circular metric space shows that there are no anchors in four-point sets.)
\begin{proof}
Given a positive integer $n$ other than $4$, consider a set $V$ of size $n$. Since the set of all three-point subsets of $V$ is an anchor in $V$, we may assume that $n\ge 10$: otherwise $\binom{n}{3}<\frac{7}{6}n^2$. We shall rely on the following theorem of Fort and Hedlund~\cite[Theorem 1]{FH58}: 
\begin{quotation}
\noindent{\em For every integer $m$ such that $m\ge 3$ and for every set $X$ of size $m$ there is a family $F$ of 
\[
\left\lceil \frac{m}{3}\left\lceil \frac{m-1}{2}\right\rceil\right\rceil
\]
three-point subsets of $X$ such that every two-point subset of $X$ is contained in a member of $F$.}
\end{quotation}
(Actually, Fort and Hedlund also proved that the size of $F$ in their theorem cannot be reduced.) We shall apply this theorem to a subset $X$ of $V$ such that $\abs{X}=n-2$. Set $Y=V-X$ and let $A$ denote the union of $F$ and all triangles in $V$ that intersect $Y$. We have 
\[ 
\abs{A}=\left\lceil \frac{n-2}{3}\left\lceil \frac{n-3}{2}\right\rceil\right\rceil+(n-2)+2\binom{n-2}{2}
=\left\lceil \frac{n-2}{3}\left\lceil \frac{n-3}{2}\right\rceil\right\rceil+(n-2)^2.
\]
It is a routine matter to verify that
\[
\left\lceil \frac{m}{3}\left\lceil \frac{m-1}{2}\right\rceil\right\rceil \le 
\left\lceil \frac{m^2}{6}\right\rceil\le \frac{m^2+2}{6}
\]
for all $m$. It follows that
\[
\abs{A}\le \frac{(n-2)^2+2}{6}+(n-2)^2<\frac{7}{6}n^2.
\]
\noindent We are going to prove that $A$ is an anchor in $V$. For this purpose, consider an arbitrary metric space $(V,\dist)$ such that all triangles in $A$ are degenerate. Our task is to show that every triangle outside $A$ is also degenerate. Since the triangle is outside $A$, it is disjoint from $Y$. Label its three points as $x,y,z$ and label the two points of $Y$ as $a,b$. 
By definition of $F$, there are (not necessarily distinct) points $u,v,w$ in $X$ such that $\{u,y,z\}\in F$, $\{x,v,z\}\in F$, $\{x,y,w\}\in F$.\\
Theorem~\ref{riri} applied to $(\{a,b,u,y,z\},\dist)$ shows that $(\{a,b,y,z\},\dist)$ is linear;\\
Theorem~\ref{riri} applied to $(\{a,b,x,v,z\},\dist)$ shows that $(\{a,b,x,z\},\dist)$ is linear;\\
Theorem~\ref{riri} applied to $(\{a,b,x,y,w\},\dist)$ shows that $(\{a,b,x,y\},\dist)$ is linear.\\
To summarize, each of the four-point subspaces of $(\{a,b,x,y,z\},\dist)$ that do not contain the triangle $\{x,y,z\}$ is linear. This, combined with Lemma~\ref{sza2}, implies that $\{x,y,z\}$ is degenerate.
\end{proof}

We do not know whether or not the coefficient $\frac{7}{6}$ in Theorem~\ref{gg} can be reduced. Nevertheless, we do know that it cannot be reduced below $\frac{1}{2}$: 
\begin{theorem}\label{peter}
Every anchor in a set of size $n$, with $n \geq 5$, includes at least $n(n-1)/2$ triangles.
\end{theorem}
\begin{proof}
Given a set $V$ of size $n$ and a set $C$ of fewer than $n(n-1)/2$ triangles in $V$, we shall find a metric space on ground set $V$ where all triangles in $C$ are degenerate, but at least one triangle is nondegenerate. For this purpose, note that there are fewer than $3n(n-1)/2$ pairs $(\{x,y\},z)$ such that $\{x,y,z\}\in C$, and so some $\{x,y\}$ belongs to fewer than $3$ of these pairs. To put it differently, there are pairwise distinct points $x,y,z_1,z_2$ in $V$ such that $\{x,y,z\}\in C$ only if $z=z_1$ or $z=z_2$.
Having enumerated the remaining points of $V$ as $z_3, z_4, \ldots z_{n-2}$, consider the metric space induced by the graph with vertex set $V$ and edges $z_1x,\; xz_2,\; z_1y,\; yz_2,\; z_2z_3,\; z_3z_4, \ldots z_{n-3}z_{n-2}$: there, a triangle is nondegenerate if and only if it is one of $\{x,y,z_3\}$, $\{x,y,z_4\}$, \ldots $\{x,y,z_{n-2}\}$. 
\end{proof}

All our theorems extend to the context of {\em pseudometric betweenness\/} defined in~\cite[p.~643]{BBC13}, which is more general than the context of metric spaces. An early draft of this note appeared as~\cite{CK22} with a discussion of {\em weakly saturated hypergraphs\/} and their relation to anchors.


\begin{thebibliography}{2}
\bibitem{ACH} Aboulker,~P., Chen,~X., Huzhang,~G., Kapadia,~R., Supko,~C. (2016). Lines, betweenness and metric spaces. {\em Discrete Comput. Geom.\/} 56 (2): 427--448.

\bibitem{BBC13} Beaudou,~L., Chv\' atal,~V., Bondy,~A., Chen,~X., Chiniforooshan,~E., Chudnovsky,~M., Fraiman,~N., Zwols,~Y. (2013). Lines in hypergraphs. {\em Combinatorica\/} 33 (6): 633--654.

\bibitem{CK22} Chv\' atal,~V., Kantor,~I. (2022). Metric spaces in which many triangles are degenerate. {\tt arXiv:2209.14361 [math.MG]}.

\bibitem{FH58} Fort Jr, M.K. and Hedlund, G.A.(1958). Minimal coverings of pairs by triples. {\em Pacific J. Mathematics\/} 8 (4): 709--719. 

\bibitem{RiRi97} Richmond,~B., Richmond,~T. (1997). Metric spaces in which all triangles are degenerate. {\em Amer. Math. Monthly\/} 104 (8): 713--719.

\bibitem{Sza22} Szab\' o, P.~G.~N. (2022). Betweenness structures of small linear co-size. {\em Discrete Appl. Math.\/} 322: 404--424. 


\end{thebibliography}
\end{document}